\documentclass[a4paper]{article}
\usepackage{amsfonts}
\usepackage{amssymb}
\usepackage{amsmath}
\usepackage{latexsym}
\usepackage{graphicx}
\usepackage{amsthm}
\usepackage{calrsfs}
 \newtheorem{thm}{Theorem}[section]
 \newtheorem{cor}[thm]{Corollary}
 \newtheorem{lem}[thm]{Lemma}
 
 \theoremstyle{definition}
 \newtheorem{defn}[thm]{Definition}
 \theoremstyle{remark}
 \newtheorem{rem}[thm]{Remark}
 \newtheorem*{ex}{Example}
 \numberwithin{equation}{section}

\begin{document}

\parbox{1mm}

\begin{center}
{\bf {\sc \Large On weighted strong type inequalities for the
generalized weighted mean operator}}
\end{center}

\vskip 12pt

\begin{center}
{\bf Ondrej HUTN\'IK}
\end{center}

%\vskip 24pt

\hspace{5mm}\parbox[t]{12cm}{\fontsize{9pt}{0.1in}\selectfont\noindent{\bf
Abstract.}  The generalized weighted mean operator
$\mathbf{M}^{g}_{w}$ is given by
$$[\mathbf{M}^{g}_{w}f](x)=
g^{-1}\left(\frac{1}{W(x)}\int_{0}^{x}w(t)g(f(t))\,\mathrm{d}t\right),$$
with $$W(x)=\int_{0}^{x} w(s)\,\mathrm{d}s, \quad \textrm{for } x
\in (0, +\infty),$$ where $w$ is a positive measurable function on
$(0,+\infty)$ and $g$ is a real continuous strictly monotone
function with its inverse $g^{-1}$. We give some sufficient
conditions on weights $u,v$ on $(0,+\infty)$ for which there
exists a positive constant $C$ such that the weighted strong type
$(p,q)$ inequality
$$\left(\int_{0}^{\infty}
u(x)\Bigl([\mathbf{M}^{g}_{w}f](x)\Bigr)^{q}\,\mathrm{d}x
\right)^{1 \over q} \leq C \left(
\int_{0}^{\infty}v(x)f(x)^{p}\,\mathrm{d}x \right)^{1 \over p}
$$ holds for every measurable non-negative function $f$, where the positive reals $p,q$ satisfy certain restrictions.
} \vskip 24pt

%%%%%%%%%%%%%%%%%%%%%%%%%%%%%%%%%%%%%%%%%%%%%%%%%%%%%%%%%%%%%%%%%%%%%%%%%%%%%%%%%%%%%%%%%%%%%%%
\section{Introduction and preliminaries}
%%%%%%%%%%%%%%%%%%%%%%%%%%%%%%%%%%%%%%%%%%%%%%%%%%%%%%%%%%%%%%%%%%%%%%%%%%%%%%%%%%%%%%%%%%%%%%%

In recent years the topic of Hardy type inequalities and their
applications seem to have grown more and more popular. Although
Hardy's original result is dated to the 1920's, some new versions
are stated and old ones are still being improved almost a century
later. One of the reasons of popularity of Hardy type inequalities
are their usefulness in various applications.

Hardy's original result was discovered in the course of attempts
to simplify the proofs of the well-known Hilbert's theorem, see
the historical part of~\cite{KMP} (cf. also~\cite{KP1}). Hardy
published his result in 1925 in paper~\cite{hardy} in the
following form:

If $p\in(1,+\infty)$ and $f$ is a non-negative function, then
\begin{equation}\label{hardy}
\int_{0}^{\infty}\Bigl([\mathbf{H}f](x)\Bigr)^{p}\,\mathrm{d}x
\leq \left(\frac{p}{p-1}\right)^{p} \int_{0}^{\infty}
f(x)^{p}\,\mathrm{d}x,
\end{equation} where $\mathbf{H}$ denotes the usual Hardy's averaging operator
\begin{equation}\label{hardyop}
[\mathbf{H}f](x):=\frac{1}{x}\int_{0}^{x} f(t)\,\mathrm{d}t, \quad
x>0.
\end{equation}The dramatic period of research until Hardy finally proved this result
was described by Kufner et al. in~\cite{KMP}. Afterwards Hardy
proved this inequality for sequences, for functions on the half
line, and for functions in Lebesgue spaces with power weights.
From the early 1970's onward a great many related results were
established under the general heading of Hardy type inequalities,
and a number of papers have been published providing new proofs,
improvements, refinements, generalizations and many applications,
see e.g.~\cite{hanjs}, \cite{heinig}, \cite{jain},
\cite{pachpatte}, and \cite{stepanov} to mention a few. Concerning
the history and development of inequality~(\ref{hardy}) we refer
the interested reader to the books~\cite{KP1} and~\cite{EE},
\cite{KP2} devoted to this subject from different viewpoints.

It is also well-known that the classical P\'olya-Knopp's
inequality, cf.~\cite{hardy-littlewood},
\begin{equation}\label{polyaknopp}
\int_{0}^{\infty} [\mathbf{G}f](x)\,\mathrm{d}x \leq \mathrm{e}
\int_{0}^{\infty}f(x)\,\mathrm{d}x,
\end{equation} where $\mathbf{G}$ denotes
the geometric mean operator defined as
\begin{equation}\label{geop}
[\mathbf{G}f](x) := \exp\left(\frac{1}{x}\int_{0}^{x}\ln
f(t)\,\mathrm{d}t\right), \quad x > 0,
\end{equation}
may be derived as a limiting case of the Hardy's
inequality~(\ref{hardy}) by changing $f \rightarrow f^{1 \over p}$
and tending $p \rightarrow \infty$, i.e.,
$$\lim_{p\rightarrow \infty}\Bigl([\mathbf{H}f^{1 \over p}](x)\Bigr)^{p}=[\mathbf{G}f](x),
\quad \textrm{and } \quad \lim_{p\rightarrow \infty}
\left(\frac{p}{p-1}\right)^{p}=\mathrm{e}.$$

In this note we propose the following general mean type inequality
problem: Let $p>0$, $0<q<+\infty$, and
\begin{equation}\label{hardyop2}
W(x):=\int_{0}^{x} w(s)\,\mathrm{d}s, \quad \textrm{for } x \in
(0, +\infty),
\end{equation}
where $w$ is a positive measurable function on $(0,+\infty)$.
\textsf{Find necessary and/or sufficient conditions on the
positive measurable functions $u, v$ (weights) and establish a
class of functions $g$ (a real continuous and strictly monotone
function with its inverse $g^{-1}$) so that the following general
mean type inequality}

\begin{equation}\label{problem}
\left(\int_{0}^{\infty}
u(x)\Bigl([\mathbf{M}^{g}_{w}f](x)\Bigr)^{q}\,\mathrm{d}x
\right)^{1 \over q} \leq C \left(
\int_{0}^{\infty}v(x)f(x)^{p}\,\mathrm{d}x \right)^{1 \over p},
\quad f \geq 0,
\end{equation}

\noindent \textsf{holds for a positive finite constant $C$}, where
\begin{equation}\label{p-genop}
[\mathbf{M}^{g}_{w}f](x):=
g^{-1}\left(\frac{1}{W(x)}\int_{0}^{x}w(t)g(f(t))\,\mathrm{d}t\right)
\end{equation}
\noindent is the \textit{generalized weighted mean operator}. In
the case of constant weight $w$ we also denote the
\textit{generalized} (non-weighted) \textit{mean operator} as
\begin{equation}\label{genop}
[\mathbf{M}^{g}f](x)=
g^{-1}\left(\frac{1}{x}\int_{0}^{x}g(f(t))\,\mathrm{d}t\right).
\end{equation}

It may be seen that the inequality~(\ref{problem}) is a natural
generalization of Hardy's inequality. In particular, in the case
$g^{-1}(x)=x$ the operator $\mathbf{M}^{g}_{w}$ reduces to the
weighted Hardy's averaging operator $\mathbf{H}_{w}$ and we have
the weighted version of Hardy's inequality~(\ref{hardy}). Putting
$g^{-1}(x)=\exp(x)$ we get the weighted geometric mean operator
$\mathbf{M}^{g}_{w}=\mathbf{G}_{w}$ and the related weighted form
of P\'olya-Knopp's inequality~(\ref{polyaknopp}). Note that the
integral operator $\mathbf{M}^{g}_{w}$ also generalizes the
harmonic mean operator, cf.~\cite{ortega}, the power mean
operator, cf.~\cite{jain}, and some other integral operators.

This paper consists of several observations (rather than
solutions!) on the stated general problem providing certain
sufficient conditions for validity of inequality~(\ref{problem}).
First, we will state some preliminary results and prove an
equivalency relation between two versions of the general mean type
inequality for the generalized weighted mean operator
$\mathbf{M}^{g}_{w}$ and its non-weighted variant
$\mathbf{M}^{g}$. Since such a reduction is possible, in the last
two sections we study only inequalities involving the integral
operator~(\ref{genop}) instead of~(\ref{p-genop}). Using certain
known methods we give some sufficient conditions for the
inequality~(\ref{problem}) to be valid.

%%%%%%%%%%%%%%%%%%%%%%%%%%%%%%%%%%%%%%%%%%%%%%%%%%%%%%%%%%%%%%%
\section{First observation: Jensen's inequality in action}
%%%%%%%%%%%%%%%%%%%%%%%%%%%%%%%%%%%%%%%%%%%%%%%%%%%%%%%%%%%%%%%

Jensen's inequality plays an important role when studying some
inequalities among different means and operators. In our notation
it has the following formulation.

\begin{lem}[Jensen's Inequality]\label{jensen}
Let $w,f$ be two non-negative integrable functions on
$(0,+\infty)$ such that $a<f(x)<b$ for all $x\in (0,+\infty)$,
where $-\infty\leq a < b \leq +\infty$.
\begin{itemize}
\item[(i)] If $g$ is a convex function on $(a, b)$, then
$g\bigl([\mathbf{H}_{w}f](x)\bigr)\leq [\mathbf{H}_{w}g\circ
f](x)$. \item[(ii)] If $g$ is a concave function on $(a, b)$, then
$[\mathbf{H}_{w}g\circ f](x) \leq
g\bigl([\mathbf{H}_{w}f](x)\bigr)$.
\end{itemize}
\end{lem}

\noindent As a direct consequence of Jensen's inequality we obtain
the following useful result. Its proof is easy and therefore
omitted.

\begin{cor}\label{corjensen}
Let $w,f$ be two non-negative integrable functions on
$(0,+\infty)$ such that $a<f(x)<b$ for all $x\in (0,+\infty)$,
where $-\infty\leq a < b \leq +\infty$.
\begin{itemize}
\item[(i)] If $g$ is a convex increasing or concave decreasing
function on $(a, b)$, then $[\mathbf{H}_{w}f](x) \leq
[\mathbf{M}^{g}_{w}f](x)$. \item[(ii)] If $g$ is a convex
decreasing or concave increasing function on $(a, b)$, then
$[\mathbf{M}^{g}_{w}f](x) \leq [\mathbf{H}_{w}f](x)$.
\end{itemize}
\end{cor}

%\proof Let $g$ be a convex increasing function. Applying the
%inverse of $g$ to both sides of Jensen's inequality~(\ref{JI}) we
%obtain the desired result~(\ref{JIM}). Proofs of remaining parts
%are similar.\qed

\noindent Immediately, a simple consequence is the following
result.

\begin{thm}\label{thm1}
Let $u$ be a weight function on $(0,+\infty)$ and let $w(x)\geq 0$
for each $x\in(0,+\infty)$. Assume that $\frac{w(t)u(x)}{W(x)}$ is
locally integrable on $(0,+\infty)$ for each fixed $t\in
(0,+\infty)$, and define the function $v$ by
$$v(t)=w(t) \int_{t}^{\infty} \frac{u(x)}{W(x)}\,\mathrm{d}x < +\infty, \quad t\in (0,+\infty).$$
If $g:(0,+\infty) \to (a,b)$, where $-\infty \leq a < b \leq
+\infty$, is either convex decreasing or concave increasing, then
$$\int_{0}^{\infty}
u(x)[\mathbf{M}^{g}_{w}f](x)\,\mathrm{d}x \leq
\int_{0}^{\infty}v(x)f(x)\,\mathrm{d}x,$$ for all $f$ such that $a
< f(x) < b$ with $x\in [0,+\infty)$.
\end{thm}

\begin{proof} By Corollary~\ref{corjensen} we get
$$\int_{0}^{\infty}u(x)[\mathbf{M}^{g}_{w}f](x)\,\mathrm{d}x \leq \int_{0}^{\infty}
u(x)[\mathbf{H}_{w}f](x)\,\mathrm{d}x.$$ Applying Fubini's theorem
we find that
\begin{align*}
\int_{0}^{\infty} u(x)[\mathbf{H}_{w}f](x)\,\mathrm{d}x & =
\int_{0}^{\infty}
f(t)\left(w(t)\int_{t}^{\infty}\frac{1}{W(x)}u(x)\,\mathrm{d}x\right)\,\mathrm{d}t
\\ & = \int_{0}^{\infty}v(x)f(x)\,\mathrm{d}x.
\end{align*} Hence the result.
\end{proof}

For the special case $w(t)\equiv 1$ we have

\begin{cor}\label{corw(t)=1}
Let $u$ be a weight function on $(0,+\infty)$ and let $v$ be
defined as $$v(t) = t
\int_{t}^{\infty}u(x)\,\frac{\mathrm{d}x}{x}, \quad t\in
(0,+\infty).$$ If $g:(0,+\infty) \to (a,b)$, where $-\infty \leq a
< b \leq +\infty$, is either convex decreasing or concave
increasing, then
\begin{equation}\label{eqw(t)=1}
\int_{0}^{\infty} u(x)[\mathbf{M}^{g}f](x)\,\mathrm{d}x \leq
\int_{0}^{\infty}v(x)f(x)\,\mathrm{d}x,
\end{equation} for all $f$ such that $a <
f(x) < b$ with $x\in [0,+\infty)$.
\end{cor}

\begin{rem}\label{rem}
The anonymous referee has pointed out that the result of
Corollary~\ref{corw(t)=1} is known even in a much more general
form. For instance, in~\cite[Theorem~2.6]{K} the case with a
general kernel $k: \Omega_1\times\Omega_2\to\mathbb{R}$, where
$(\Omega_1,\mu_1)$, $(\Omega_, \mu_2)$ are measure spaces with
$\sigma$-finite measures, and the integral operator
$$[\mathbf{A}_k f](x):=\frac{1}{K(x)}\int_{\Omega_2} k(x,y)
f(y)\,\mathrm{d}\mu_2(y),$$ with $K(x):=\int_{\Omega_2}
k(x,y)\,\mathrm{d}\mu_2(y),\,\, x\in\Omega_1$, is considered. Then
the Hardy type inequality
$$\int_{\Omega_1} u(x)\Phi\bigg([\mathbf{A}_k h](x)\bigg)\,\mathrm{d}\mu_1(x) \leq \int_{\Omega_2} v(y)\Phi(f(y))\,\mathrm{d}\mu_2(y)$$
is proved therein under some natural conditions on weights $u,v$
and convexity of $\Phi$. For further details and similar results
of this type we refer to paper~\cite{KPP}. Indeed, under the
conditions of Corollary~\ref{corw(t)=1} the
inequality~(\ref{eqw(t)=1}) is equivalent to
$$\int_{0}^{\infty} u(x)\Phi\bigg([\mathbf{H}h](x)\bigg)\,\frac{\mathrm{d}x}{x} \leq \int_{0}^{\infty} v(x)\Phi(f(x))\,\frac{\mathrm{d}x}{x},$$
when replacing $g\circ f$ by $h$, $g^{-1}$ by $\Phi$, $u(x)$ by
$\frac{u(x)}{x}$ and $v(x)$ by $\frac{v(x)}{x}$, which corresponds
to the "weighted" analog of Hardy-Knopp type inequality,
cf.~\cite{KPO}.
\end{rem}

\begin{ex}
Choosing the function $g(x)=\ln x$ and replacing $f$ by $f^{p}$
with $p>0$ in~(\ref{eqw(t)=1}) we get the following P\'olya-Knopp
type inequality
$$\int_{0}^{\infty}u(x)\bigg([\mathbf{G}f](x)\bigg)^{p}\,\mathrm{d}x
\leq \int_{0}^{\infty}v(x)f(x)^{p}\,\mathrm{d}x,$$ where $u$ and
$v$ are defined as in Corollary~\ref{corw(t)=1}.
\end{ex}

%%%%%%%%%%%%%%%%%%%%%%%%%%%%%%%%%%%%%%%%%%%%%%%%%%%%%%%%%%%%%%%
\section{Second observation: the reduction lemma}\label{secreduction}
%%%%%%%%%%%%%%%%%%%%%%%%%%%%%%%%%%%%%%%%%%%%%%%%%%%%%%%%%%%%%%%

If we want to consider the general mean type
inequality~(\ref{problem}) with weights we propose to reduce the
weighted operator $\mathbf{M}_w^g$  into its non-weighted variant
$\mathbf{M}^g$, and then solve this inequality in a reduced form.

\begin{lem}\label{reduction}
Let $0<p$ and $q<+\infty$. Let $u$ and $v$ be two weight functions
on $(0, +\infty)$, $f$ be a positive function on $(0,+\infty)$ and
$g$ be a real continuous and strictly monotone function. Moreover,
let $w$ be a strictly positive function on $(0,+\infty)$ and $W$
be defined as in~(\ref{hardyop2}) such that $W(+\infty)=+\infty$.
Then the inequality
\begin{equation}\label{red1}
\left(\int_{0}^{\infty}
u(x)\Bigl([\mathbf{M}^{g}_{w}f](x)\Bigr)^{q}\,\mathrm{d}x
\right)^{1 \over q} \leq C
\left(\int_{0}^{\infty}v(x)f(x)^{p}\,\mathrm{d}x \right)^{1 \over
p}
\end{equation}
holds if and only if the inequality
\begin{equation}\label{red2}
\left(\int_{0}^{\infty}
U(x)\Bigl([\mathbf{M}^{g}f](x)\Bigr)^{q}\,\mathrm{d}x \right)^{1
\over q} \leq C \left(\int_{0}^{\infty}V(x)f(x)^{p}\,\mathrm{d}x
\right)^{1 \over p}
\end{equation}
holds with the same positive finite constant $C$ and weights
\begin{equation}\label{U,V weights}
U(x):=\frac{u\left(W^{-1}(x)\right)}{W'\left(W^{-1}(x)\right)},
\quad V(x):=
\frac{v\left(W^{-1}(x)\right)}{W'\left(W^{-1}(x)\right)}.
\end{equation}
\end{lem}

\begin{proof} Considering the generalized mean
operators~(\ref{p-genop}) and~(\ref{genop}), we have that
$$[\mathbf{M}^{g}_{w}f](x) = [\mathbf{M}^{g}h](W(x)),$$ where
$h(y)=f\left(W^{-1}(y)\right)$. Now, the inequality~(\ref{red1})
reads
$$\left(\int_{0}^{\infty} u(x)
\Bigl([\mathbf{M}^{g}h](W(x))\Bigr)^{q}\,\mathrm{d}x\right)^{1
\over q} \leq C \left(\int_{0}^{\infty}
v(x)\Bigl(h(W(x))\Bigr)^{p}\,\mathrm{d}x\right)^{1 \over p}.$$
Using the substitution $y=W(x)$ the last inequality is equivalent
to~(\ref{red2}).
\end{proof}

\begin{rem}
Lemma~\ref{reduction} was used for the first time in
paper~\cite{jain} in the context of the weighted P\'olya-Knopp's
inequality, i.e., $g^{-1}(x)=\exp(x)$. In principle, lemma says
that the inequality~(\ref{red1}) is not more general than the
inequality~(\ref{red2}). Also, for the case when $w$ is a
continuous and strictly positive function, then the inequalities
of the type~(\ref{red1}) may be obtained by only studying the
basic inequality~(\ref{red2}). So for this reason we will study
some inequalities related to the inequality~(\ref{red2}) only.
\end{rem}

\begin{ex}(cf.~\cite{wedestig})
Let $0 < p \leq q < +\infty$, $\lambda > 0$, and $u(x)=x^{r}$,
$v(x)=x^{s}$, where $r,s \in \mathbb{R}$ such that
$\frac{r+1}{q}=\frac{s+1}{p}$. When $g^{-1}(x)=\exp(x)$ and
$w(t)=t^{\lambda-1}$, we have the Cochran-Lee type inequality,
cf.~\cite{jain} and \cite{jain3},
$$\left(\int_{0}^{\infty}
x^{r}\left[\exp\left(\frac{\lambda}{x^{\lambda}}\int_{0}^{x}t^{\lambda
-1}\ln f(t)\,\mathrm{d}t\right)\right]^{q}\,\mathrm{d}x \right)^{1
\over q} \leq C \left(\int_{0}^{\infty}x^{s}f(x)^{p}\,\mathrm{d}x
\right)^{1 \over p},
$$ for a positive finite constant $C$. According to the
reduction lemma and notation~(\ref{geop}) this inequality is
equivalent to
$$\left(\int_{0}^{\infty} U(x) \Bigl([\mathbf{G}h](x)\Bigr)^{q}\,\mathrm{d}x\right)^{1
\over q} \leq C \left(\int_{0}^{\infty}
V(x)h(x)^{p}\,\mathrm{d}x\right)^{1 \over p},$$ where
$$h(x)=f\left((\lambda x)^{1 \over \lambda}\right), \quad
U(x)=(\lambda x)^{\frac{r+1}{\lambda}-1}, \quad V(x)=(\lambda
x)^{\frac{s+1}{\lambda}-1},$$ which may be rewritten to the form
$$\left(\int_{0}^{\infty} (\lambda
x)^{\frac{r+1}{\lambda}-1}\Bigl([\mathbf{G}h](x)\Bigr)^{q}\,\mathrm{d}x
\right)^{1 \over q} \leq C \left(\int_{0}^{\infty}(\lambda
x)^{\frac{s+1}{\lambda}-1}h(x)^{p}\,\mathrm{d}x \right)^{1 \over
p}.$$
\end{ex}

%%%%%%%%%%%%%%%%%%%%%%%%%%%%%%%%%%%%%%%%%%%%%%%%%%%%%%%%%%%%%%%%%%%%%%%%%%%%%%%%
\section{Third observation: Levinson's approach}\label{secresults}
%%%%%%%%%%%%%%%%%%%%%%%%%%%%%%%%%%%%%%%%%%%%%%%%%%%%%%%%%%%%%%%%%%%%%%%%%%%%%%%%

As already mentioned, an elementary approach to the stated problem
consists of using the notion of convexity. Indeed, according to
Corollary~\ref{corjensen} we have that if $g$ is a convex
decreasing or concave increasing function, then
$$[\mathbf{M}^{g}_{w}f](x) \leq [\mathbf{H}_{w}f](x)$$ for all admissible functions $f$ on $(0,+\infty)$. Then
sufficient conditions for the weighted Hardy's inequality (for
which the problem is solved by many authors) are also sufficient
for the inequality~(\ref{problem}) to hold. However, this is quite
a rough approach.

Note that if $g$ is a strictly monotone function then replacing
$g^{-1}$ by $\varphi$ and $g(f)$ by $h$ and using
Lemma~\ref{reduction}, the inequality~(\ref{problem}) may be
rewritten as follows:
\begin{equation}\label{levinson}
\left(\int_{0}^{\infty}U(x)\Bigl(\varphi([\mathbf{H}h](x))\Bigr)^{q}\,\mathrm{d}x\right)^{1
\over q} \leq C
\left(\int_{0}^{\infty}V(x)\Bigl(\varphi(h(x))\Bigr)^{p}\right)^{1
\over p},
\end{equation} where $U(x), V(x)$ are given in~(\ref{U,V weights}) for $w$ a strictly
positive function on $(0,+\infty)$. Clearly, if $g$ is either a
convex decreasing or concave increasing function, then $\varphi$
is a convex function.

The inequality~(\ref{levinson}) is in fact the weighted extension
of Levinson's modular inequality which was studied in non-weighted
form for $N$-functions in~\cite{levinson}. Its weighted form for
the case $p=q=1$ was proved by Hans P. Heinig in~\cite{heinig2}.
In this section we will study the inequality~(\ref{levinson}) in a
general case for $1 < p \leq q < +\infty$ and we will prove some
of its modifications. For this purpose we will consider the
following classes of functions, cf.~\cite{levinson}.

\begin{defn}\label{fciaphi}
A function $\varphi: (a,b) \rightarrow (0,+\infty)$, where $0 \leq
a < b \leq +\infty$, belongs to the class $\Phi_{r}$, $r>1$, if
$$\varphi(x)\varphi''(x) \geq
\left(1-\frac{1}{r}\right)\left[\varphi'(x)\right]^{2}$$ holds for
all $x>0$. If $r=+\infty$, we write $\Phi_{\infty}=\Phi$.
\end{defn}

The usual examples of functions belonging to the class $\Phi$ are
the Euler Gamma function $\Gamma$ as well as functions
$\varphi_1(x)=x^{-a}$ for $a>0$, and
$\varphi_2(x)=\mathrm{e}^{x^b}$ for $b\geq 1$. The inclusion $\Phi
\subset \Phi_{r}$ is strict, because when choosing
$\varphi(x)=x^{s}$ for $s \geq r$, then $\varphi \in
\Phi_{r}\setminus \Phi$. Thus, $\varphi_1$ and $\varphi_2$ are in
the class $\Phi_r$ for each $r>1$. However, for $b\in(0,1)$ we
have $\varphi_2\notin\Phi_r$ for any $r>1$.

It is also easy to verify that for $r>1$ when $\varphi \in
\Phi_{r}$ the function $\psi=\varphi^{1 \over r}$ is convex,
whereas for $\varphi \in \Phi$ the function $\psi=\ln \varphi$ is
convex. This enables us to state the following theorem which is a
generalization of the weighted extension of Levinson's
result~\cite{levinson}. Observe that the well-known weight
condition of Muckenhoupt, cf.~\cite{muckenhoupt}, is used here.

\begin{thm}\label{theorem1}
Let $1 < p \leq q < +\infty$, and $p':=\frac{p}{p-1}$.
\begin{itemize}
\item[(a)] If $\varphi \in \Phi_{q}$ and
\begin{equation}\label{muckenhoupt}
\sup_{\tau>0}\left(\int_{\tau}^{\infty}\frac{U(x)}{x^{q}}\,\mathrm{d}x\right)^{1
\over q} \left(\int_{0}^{\tau}V(x)^{1-p'}\right)^{1 \over p'} <
+\infty,
\end{equation}
then
\begin{equation}\label{eqthm1}
\left(\int_{0}^{\infty}U(x)\varphi([\mathbf{H}h](x))\,\mathrm{d}x\right)^{1
\over q} \leq C
\left(\int_{0}^{\infty}V(x)\varphi(h(x))\,\mathrm{d}x\right)^{1
\over p}.
\end{equation}

\item[(b)] Let $s>0$, and $\varphi \in \Phi$. If
$$V(x)=x^{\lambda} \int_{x}^{\infty}
\frac{U(t)}{t^{\lambda+1}}\,\mathrm{d}t, \quad
\textrm{for}\,\,\,\lambda>0,$$ then
$$\int_{0}^{\infty}U(x)\Bigl(\varphi([\mathbf{H}h](x))\Bigr)^{s}\,\mathrm{d}x \leq
C \int_{0}^{\infty} V(x)
\Bigl(\varphi(h(x))\Bigr)^{s}\,\mathrm{d}x
$$ with constant $C=\mathrm{e}^{\lambda}$.
\end{itemize}
\end{thm}

\begin{proof} First we prove the part~(a). Since $\psi=\varphi^{1
\over q}$ is convex, then by Jensen's inequality we have
\begin{align*}
\left(\int_{0}^{\infty}U(x)\varphi([\mathbf{H}h](x))\,\mathrm{d}x\right)^{1
\over q} & =
\left(\int_{0}^{\infty}U(x)\Bigl(\psi([\mathbf{H}h](x)\Bigr)^{q}\,\mathrm{d}x\right)^{1
\over q}
\\ & \leq
\left(\int_{0}^{\infty}U(x)\Bigl([\mathbf{H}\psi(h)](x)\Bigr)^{q}\,\mathrm{d}x\right)^{1
\over q}.
\end{align*}By the well-known Muckenhoupt's weight condition~(\ref{muckenhoupt}) we obtain
\begin{equation}\label{eq2thm1}
\left(\int_{0}^{\infty}U(x)\Bigl([\mathbf{H}\psi(h)](x)\Bigr)^{q}\,\mathrm{d}x\right)^{1
\over q} \leq C \left(
\int_{0}^{\infty}V(x)\Bigl(\psi(h(x))\Bigr)^{p}\,\mathrm{d}x
\right)^{1 \over p}.
\end{equation}
For $1<p\leq q < +\infty$ and $\varphi \in \Phi_{q}$ we have
$\varphi \in \Phi_{p}$, because
$$\varphi(x)\varphi''(x) \geq
\left(1-\frac{1}{q}\right)\left[\varphi'(x)\right]^{2} \geq
\left(1-\frac{1}{p}\right)\left[\varphi'(x)\right]^{2}
$$ holds for all $x>0$. Therefore $\psi(h)^{p}=\varphi(h)$ and the
inequality~(\ref{eqthm1}) is proved.

For part~(b) we apply Jensen's inequality for the convex function
$\psi=\ln\varphi$ to get
\begin{align*}
\int_{0}^{\infty}U(x)\Bigl(\varphi([\mathbf{H}h](x))\Bigr)^{s}\,\mathrm{d}x
& = \int_{0}^{\infty}U(x)\left[\exp
\Bigl(\psi([\mathbf{H}h](x))\Bigr)\right]^{s}\,\mathrm{d}x \\ &
\leq \int_{0}^{\infty}U(x)\left[\exp
\Bigl([\mathbf{H}\psi(h)](x)\Bigr)\right]^{s}\,\mathrm{d}x \\
& = \int_{0}^{\infty}U(x)
\Bigl([\mathbf{G}\varphi(h)](x)\Bigr)^{s}\,\mathrm{d}x.
\end{align*}Using the fact $[\mathbf{G}f]^{s}=[\mathbf{G}f^{s}]$ and
substitution $t=xy$, we have
\begin{align*}
\int_{0}^{\infty}U(x)
\Bigl([\mathbf{G}\varphi(h)](x)\Bigr)^{s}\,\mathrm{d}x & =
\int_{0}^{\infty}U(x) [\mathbf{G}(\varphi(h))^{s}](x)\,\mathrm{d}x
\\ & =  \int_{0}^{\infty}U(x) \exp\left(\int_{0}^{1}\ln
\Bigl(\varphi(h(xy))\Bigr)^{s}\,\mathrm{d}y\right)\,\mathrm{d}x.
\end{align*}Since
$$-\lambda = \int_{0}^{1}\ln y^{\lambda}\,\mathrm{d}y, \quad \textrm{for } \lambda > 0,$$
we obtain
\begin{align*}
& {\phantom{1}} \int_{0}^{\infty}U(x) \exp\left(\int_{0}^{1}\ln
\Bigl(\varphi(h(xy))\Bigr)^{s}\,\mathrm{d}y\right)\,\mathrm{d}x \\
& = \mathrm{e}^{\lambda} \int_{0}^{\infty}U(x)
\exp\left(\int_{0}^{1}\ln \left[y^{\lambda}\cdot
\Bigl(\varphi(h(xy))\Bigr)^{s}\right]\,\mathrm{d}y\right)\,\mathrm{d}x.
\end{align*}Applying Jensen's inequality again and interchanging the order of
integration we get
\begin{align*}
& {\phantom{1}} \mathrm{e}^{\lambda} \int_{0}^{\infty}U(x)
\exp\left(\int_{0}^{1}\ln y^{\lambda}\cdot
\Bigl(\varphi(h(xy))\Bigr)^{s}\,\mathrm{d}y\right)\,\mathrm{d}x
\\ & \leq \mathrm{e}^{\lambda} \int_{0}^{\infty}U(x)
\left(\int_{0}^{1}y^{\lambda}
\Bigl(\varphi(h(xy))\Bigr)^{s}\,\mathrm{d}y\right)\,\mathrm{d}x
\\ & =  \mathrm{e}^{\lambda} \int_{0}^{1}y^{\lambda}
\left(\int_{0}^{\infty}
U(x)\Bigl(\varphi(h(xy))\Bigr)^{s}\,\mathrm{d}x\right)\,\mathrm{d}y.
\end{align*}Substituting $t=xy$ and using Fubini's theorem one has
\begin{align*}
& {\phantom{1}} \mathrm{e}^{\lambda} \int_{0}^{1}y^{\lambda -1}
\left(\int_{0}^{\infty}
U\left(\frac{t}{y}\right)\Bigl(\varphi(h(t))\Bigr)^{s}\,\mathrm{d}t\right)\,\mathrm{d}y
\\ & =  \mathrm{e}^{\lambda}
\int_{0}^{\infty}\Bigl(\varphi(h(t))\Bigr)^{s}
\left(\int_{0}^{1}y^{\lambda -1}
U\left(\frac{t}{y}\right)\,\mathrm{d}y\right)\,\mathrm{d}t\\ & =
\mathrm{e}^{\lambda}
\int_{0}^{\infty}\Bigl(\varphi(h(t))\Bigr)^{s} \left(t^{\lambda}
\int_{t}^{\infty}\frac{U(x)}{x^{\lambda+1}}\,\mathrm{d}x\right)\,\mathrm{d}t.
\end{align*}Hence the result.
\end{proof}

Recall that the condition~(\ref{muckenhoupt}) was first
established in B.~Muckenhoupt's paper~\cite{muckenhoupt} for
$p=q$. Moreover, the condition~(\ref{muckenhoupt}) is necessary
and sufficient for~(\ref{eq2thm1}).

From the proof of Theorem~\ref{theorem1} (b) we immediately have
that for $\varphi\in\Phi$ the inequality
$$\left(\int_{0}^{\infty}U(x)\Bigl(\varphi([\mathbf{H}h](x))\Bigr)^{q}\,\mathrm{d}x
\right)^{1 \over q} \leq  \left(\int_{0}^{\infty}U(x)
\Bigr([\mathbf{G}(\varphi(h)](x)\Bigr)^{q}\,\mathrm{d}x\right)^{1
\over q}$$ holds which means that we may deal with the classical
non-weighted geometric mean operator and therefore we may use the
following well-known result for $[\mathbf{G}\varphi(h)]$ under the
condition
\begin{equation}\label{eqthm2}
\sup_{x>0} x^{-\frac{1}{p}}\left(\int_{0}^{x}U(x)
\left(\left[\mathbf{G}\left(\frac{1}{V}\right)\right](x)\right)^{q
\over p}\,\mathrm{d}t\right)^{1 \over q} < +\infty,
\end{equation} whenever $0< p\leq q < +\infty$, cf.~\cite{persson}, which implies the inequality
\begin{equation}\label{eq3thm2}
\left(\int_{0}^{\infty}U(x)
\Bigr([\mathbf{G}(\varphi(h)](x)\Bigr)^{q}\,\mathrm{d}x\right)^{1
\over q} \leq C \left(\int_{0}^{\infty}
V(x)\Bigr(\varphi(h(x))\Bigl)^{p}\,\mathrm{d}x\right)^{1 \over p}.
\end{equation}
Summarizing the above we get

\begin{thm}\label{thmPhi}
Let $0< p\leq q < +\infty$, and $w$ be a strictly positive
function on $(0,+\infty)$ with $W(+\infty)=+\infty$. Let $f, g$ be
non-negative functions on $(0,+\infty)$ and, moreover, let $g$ be
real continuous and strictly monotone on $(0,+\infty)$ such that
$g^{-1} \in \Phi$. If $u,v$ are weights on $(0,+\infty)$ with $U,
V$ given by~(\ref{U,V weights}) satisfying the
condition~(\ref{eqthm2}), then the inequality~(\ref{problem})
holds with a positive finite constant $C$.
\end{thm}

%%%%%%%%%%%%%%%%%%%%%%%%%%%%%%%%%%%%%%%%%%%%%%%%%%%%%%%%%%%%%%%%%%%%%%%%%%%%%%%%
\section{Fourth observation: Wedestig's approach}
%%%%%%%%%%%%%%%%%%%%%%%%%%%%%%%%%%%%%%%%%%%%%%%%%%%%%%%%%%%%%%%%%%%%%%%%%%%%%%%%

Using the approach from~\cite{wedestig2} we have the following
similar result as Theorem~2.1 therein (observe only the sufficient
condition in our case as it is demonstrated in Example~\ref{ex}).

\begin{thm}
Let $1 < p \leq q < +\infty$, $s \in (1,p)$, and $g$ be a real
continuous either convex decreasing or concave increasing function
on $(a,b)$, where $-\infty \leq a < b \leq +\infty$. Put
\begin{equation}\label{A(s)}
A(s):=\sup_{t>0} \widetilde{V}(t)^{\frac{q(s-1)}{p}}
\left(\int_{t}^{\infty}U(x)\widetilde{V}(x)^{\frac{q(p-s)}{p}}\,\frac{\mathrm{d}x}{x^{q}}\right)^{1/q},
\end{equation} where $U, V$ are given by~(\ref{U,V weights}) for
weights $u,v$ on $(0,+\infty)$, $w$ is a strictly positive
function on $(0,+\infty)$ with $W(+\infty)=+\infty$, and
$$\widetilde{V}(t):=\int_{0}^{t}V(y)^{1-p'}\,\mathrm{d}y.$$ If $A(s)<
+\infty$, then the inequality~(\ref{problem}) holds for all $f$
such that $a < f(x) < b$ with $x\in [0,+\infty)$. Moreover, if $C$
is the best possible constant in~(\ref{problem}), then
\begin{equation}\label{Cbestposs}
C \leq \inf_{1 < s < p} \left(\frac{p-1}{p-s}\right)^{1/p'} A(s).
\end{equation}
\end{thm}

\begin{proof} Using Lemma~\ref{reduction}, replacing $g(f)$ by $h$
and $g^{-1}$ by $\varphi$ we have the inequality~(\ref{levinson}).
Then applying Jensen's inequality to the left-hand side
of~(\ref{levinson}) we get
$$
\left(\int_{0}^{\infty}U(x)\Bigl(\varphi([\mathbf{H}h](x))\Bigr)^{q}\,\mathrm{d}x\right)^{1
\over q} \leq
\left(\int_{0}^{\infty}U(x)\Bigl([\mathbf{H}\varphi(h)](x)\Bigr)^{q}\,\mathrm{d}x\right)^{1
\over q}.$$ It is clear that if we prove the estimate
$$\left(\int_{0}^{\infty}U(x)\left(\frac{1}{x}\int_{0}^{x} \varphi(h(t))\,\mathrm{d}t
\right)^{q}\,\mathrm{d}x \right)^{1 \over q} \leq
\left(\int_{0}^{\infty}V(x)\Bigl(\varphi(h(x))\Bigr)^{p}\,\mathrm{d}x\right)^{1
\over p},$$ then we get the upper estimate in~(\ref{levinson}) and
consequently in~(\ref{problem}). For this purpose put
$$V(x)\Bigl(\varphi(h(x))\Bigr)^{p}=\varphi(k(x)).$$
Now the inequality~(\ref{levinson}) takes the form
\begin{equation}\label{levinsonmod}
\left(\int_{0}^{\infty}U(x)\left(\frac{1}{x}\int_{0}^{x}
V(t)^{-\frac{1}{p}} \Bigl(\varphi(k(t))\Bigr)^{1 \over
p}\,\mathrm{d}t \right)^{q}\,\mathrm{d}x \right)^{1 \over q} \leq
C \left(\int_{0}^{\infty} \varphi(k(x))\,\mathrm{d}x\right)^{1
\over p}.
\end{equation} Using H\"{o}lder's inequality with indices $p$ and $p'$ for the left-hand side integral
in~(\ref{levinsonmod}), we obtain \small{
\begin{align*}
& {\phantom{1}}
\left(\int_{0}^{\infty}U(x)\left(\frac{1}{x}\int_{0}^{x}
V(t)^{-\frac{1}{p}} \Bigl(\varphi(k(t))\Bigr)^{1 \over
p}\,\mathrm{d}t \right)^{q}\,\mathrm{d}x \right)^{1 \over q} \\ &
= \left(\int_{0}^{\infty}U(x)\left(\frac{1}{x}\int_{0}^{x}
\Bigl(\varphi(k(t))\Bigr)^{1 \over p}
\widetilde{V}(t)^{\frac{s-1}{p}} \widetilde{V}(t)^{-\frac{s-1}{p}}
V(t)^{-\frac{1}{p}} \,\mathrm{d}t \right)^{q}\,\mathrm{d}x
\right)^{1 \over q} \\ & \leq
\left(\int_{0}^{\infty}U(x)\left(\int_{0}^{x} \varphi(k(t))
\widetilde{V}(t)^{s-1}\,dt\right)^{q \over p} \left(\int_{0}^{x}
\widetilde{V}(t)^{-\frac{p'(s-1)}{p}}
V(t)^{-\frac{p'}{p}}\,\mathrm{d}t \right)^{q \over
p'}\,\frac{\mathrm{d}x}{x^{q}} \right)^{1 \over q}.
\end{align*}}\normalsize Using the fact that $V(t)^{-p'/p} = V(t)^{1-p'}$ the last integral is equivalent to
\small{
\begin{align}\label{int}
& {\phantom{1}} \left(\frac{p}{p-(s-1)p'}\right)^{1 \over
p'}\left(\int_{0}^{\infty}U(x)\widetilde{V}(x)^{\frac{p-(s-1)p'}{p}\cdot
\frac{q}{p'}}\left(\int_{0}^{x}
\varphi(k(t)) \widetilde{V}(t)^{s-1}\,\mathrm{d}t\right)^{q \over p}\,\frac{\mathrm{d}x}{x^{q}} \right)^{1 \over q} \nonumber \\
& = \left(\frac{p-1}{p-s}\right)^{1 \over
p'}\left(\int_{0}^{\infty}U(x)\widetilde{V}(x)^{\frac{q(p-s)}{p}}\left(\int_{0}^{x}
\varphi(k(t)) \widetilde{V}(t)^{s-1}\,\mathrm{d}t\right)^{q \over
p}\,\frac{\mathrm{d}x}{x^{q}} \right)^{1 \over q}.
\end{align}}\normalsize By Minkowski's inequality the integral~(\ref{int}) is dominated by
\begin{align*}
& {\phantom{1}} \left(\frac{p-1}{p-s}\right)^{1 \over
p'}\left(\int_{0}^{\infty} \varphi(k(t))
\widetilde{V}(t)^{s-1}\left(\int_{t}^{\infty}
U(x)\widetilde{V}(x)^{\frac{q(p-s)}{p}}\,\frac{\mathrm{d}x}{x^{q}}\right)^{p
\over q}\,\mathrm{d}t \right)^{1 \over p} \\ & \leq
\left(\frac{p-1}{p-s}\right)^{1 \over p'} A(s)
\left(\int_{0}^{\infty}\varphi(k(t))\,\mathrm{d}t\right)^{1 \over
p}.
\end{align*}Hence~(\ref{levinsonmod}), thus~(\ref{levinson}) and consequently by Lemma~\ref{reduction}
the inequality~(\ref{problem}) holds with a constant satisfying
the right hand side inequality in~(\ref{Cbestposs}).
\end{proof}

The following example was suggested by Prof. Alois Kufner via
e-mail communication with the author. It demonstrates that the
condition~(\ref{A(s)}) is not necessary for the
inequality~(\ref{problem}) to hold.

\begin{ex}\label{ex}
Let $g(x) = \sqrt{x}$, $U(x) = V(x) = x$ on $[0,+\infty)$ and
$p=q=2$. Then the inequality~(\ref{problem}) holds, but the
condition~(\ref{A(s)}) is not satisfied.
\end{ex}

%%%%%%%%%%%%%%%%%%%%%%%%%%%%%%%%%%%%%%%%%%%%%%%%%%%%%%%%%%%%%%%%%%%%%%%%%%%%%%%%%%
\section*{Concluding remarks}
%%%%%%%%%%%%%%%%%%%%%%%%%%%%%%%%%%%%%%%%%%%%%%%%%%%%%%%%%%%%%%%%%%%%%%%%%%%%%%%%%%

There are many various sufficient and/or necessary conditions
known in the literature for the strong type $(p,q)$ inequalities
involving the usual Hardy's operator, or the geometric mean
operator, see the reference list below. As far as we know, the
situation for weighted inequalities with other kind of means
different from the arithmetic and geometric one is far less known,
see the papers~\cite{ortega} and~\cite{jain} dealing with the
weighted Hardy type inequalities for the (non-weighted) harmonic
mean operator $$[\mathbf{T}f](x)= \frac{1}{x}\int_{0}^{x}
\frac{t}{f(t)}\,\mathrm{d}t$$ and the (non-weighted) power mean
operator $$[\mathbf{P}_\alpha f](x) =
\left(\frac{1}{x}\int_{0}^{x}
f^\alpha(t)\,\mathrm{d}t\right)^{\frac{1}{\alpha}}, \quad
\alpha>0,$$ respectively. All these means (i.e., arithmetic,
geometric, harmonic, power) are included in the family of
quasi-arithmetic means -- a construction which resembles our
definition of operator $\mathbf{M}^g_w$. Thus, $\mathbf{M}^g_w$
might be naturally called the \textit{weighted quasi-arithmetic
mean operator} (in the "continuous", i.e., integral form).

In this paper we have discussed some elementary direct methods to
find sufficient conditions for the general mean type
inequality~(\ref{problem}) to be valid. However, the obtained
results do not take the "mean function" $g$ into account, i.e.,
Muckenhoupt's as well as Wedestig's condition is, in fact,
independent on a particular choice of $g$. On the other hand,
Theorem~\ref{thmPhi} restricts the choice of "mean function" $g$
to the case $g^{-1}\in \Phi$ of Levinson's function class, but the
proof uses a certain transformation of the general mean operator
$\mathbf{M}^g$ to the geometric mean operator $\mathbf{G}$. Thus,
the obtained conditions do not seem to be the correct ones in
order to be able to find necessary and sufficient conditions for
the stated problem. We do suppose that the "ideal" conditions
should contain the function $g$. This motivates and encourages
development of new methods to deal with the problem of finding
sufficient, as well as necessary conditions for the validity of
inequality~(\ref{problem}) in its general form proposed in this
paper.

Since the operator $\mathbf{M}^g$ covers the classical integral
operators, immediately many questions naturally arise in
connection with this operator: to provide its boundedness,
compactness, continuity criteria and other properties on
appropriate function spaces when choosing different "mean
function" $g$. However, these (as well as other related) questions
wait for their development in the future.

\vspace{0.5cm}{\bf Acknowledgements:} We are grateful to the
referee for her/his suggestions which essentially improved
previous versions of the manuscript, especially for drawing the
author's attention to recent work of Kruli\'{c}~\cite{K} on the
topic of Hardy's inequality.

\vspace{5mm}

\noindent \small{\textsc{Ondrej Hutn\'ik}
\newline Institute of Mathematics, Faculty of Science, Pavol Jozef
\v Saf\'arik University in Ko\v sice,
\newline {\it Current address:} Jesenn\'a 5, SK 040~01 Ko\v sice,
Slovakia, \newline \phantom{{\it E-mail addresses:}}
ondrej.hutnik@upjs.sk

% ------------------------------------------------------------------------

\begin{thebibliography}{1}
\bibitem{EE}
D.~E. Edmunds, W.~D. Evans, \textit{Hardy Operators, Function
Spaces and Embeddings}. Springer, Berlin, 2004.

\bibitem{hardy}
G.~H. Hardy, \textit{Notes on some points in the integral
calculus}. Mess. Math. {\bf 54} (1925), 150--156.

\bibitem{hardy-littlewood}
G.~H. Hardy, J.~E. Littlewood, G. P\'olya, \textit{Inequalities}.
Cambridge University Press, 1967.

\bibitem{hanjs}
\v Z. Hanj\v s, E.~R. Love, J. Pe\v cari\'c, \textit{On some
inequalities related to the Hardy's integral inequality}. Math.
Ineq. Appl. {\bf 4}(3) (2001), 357--368.

\bibitem{heinig}
H.~P. Heinig, \textit{Some Extensions of Hardy's Inequality}. SIAM
J. Math. Anal. {\bf 6} (1975), 698--713.

\bibitem{heinig2}
H.~P. Heinig, \textit{Modular inequalities for the Hardy averaging
operator}. Math. Bohem. {\bf 124}(2-3) (1999), 231--244.

\bibitem{jain}
P. Jain, L.~E. Persson, A. Wedestig, \textit{Carleman-Knopp type
inequalities via Hardy inequalities}. Math. Ineq. Appl. {\bf 4}(3)
(2001), 343--355.

\bibitem{jain3}
P. Jain, L.~E. Persson, A. Wedestig, \textit{Multidimensional
Cochran and Lee type inequalities with weights}. Proc. A. Razmadze
Math. Inst. {\bf 129} (2002), 17--27.

\bibitem{KPO}
S. Kaijser, L.~E. Persson, A. \"{O}berg, \textit{On Carleman and
Knopp's inequalities}. J. Approx. Theory {\bf 117}(1) (2002),
140--151.

\bibitem{KMP}
A. Kufner, L. Maligranda, L.~E. Persson, \textit{The prehistory of
the Hardy inequality}. Amer. Math Mon. {\bf 113}(8) (2006),
715--732.

\bibitem{KP1}
A. Kufner, L. Maligranda, L.~E. Persson, \textit{The Hardy
Inequality. About its History and Some Related Results}.
Vydavatelsk\'y Servis, Plze\v{n}, 2007.

\bibitem{KP2}
A. Kufner, L.~E. Persson, \textit{Weighted Inequalities of Hardy
Type}. World Scientific, New Jersey-London-Singapore-Hong Kong,
2003.

\bibitem{K}
K. Kruli\'{c}, \textit{Generalizations and refinements of Hardy's
inequality}. PhD thesis, Dept. of Math., University of Zagreb,
Croatia, 2010.

\bibitem{KPP}
K. Kruli\'{c}, J. Pe\v cari\'c, L.~E. Persson, \textit{Some new
Hardy type inequality with general kernels}. Math. Inequal. Appl.
{\bf 12}(9) (2009), 473--485.

\bibitem{levinson}
N. Levinson, \textit{Generalizations of an inequality of Hardy}.
Duke J. Math. {\bf 31} (1964), 389--394.

\bibitem{muckenhoupt}
B. Muckenhoupt, \textit{Hardy's inequality with weights}. Studia
Math. {\bf 44} (1972), 31--38.

\bibitem{ortega}
P. Ortega Salvador, C. Ram\'irez Torreblanca, \textit{Weighted
inequalities for harmonic means}. Math. Ineq. Appl. {\bf 9}(3)
(2006), 407--420.

\bibitem{pachpatte}
B.~G. Pachpatte, E.~R. Love, \textit{On some new integral
inequalities related to Hardy's integral inequality}. J. Math.
Anal. Appl. {\bf 149}(1) (1990), 17--25.

\bibitem{persson}
L.~E. Persson, V.~D. Stepanov, \textit{Weighted integral
inequalities with the geometric mean operator}. J. Ineq. Appl.
{\bf 7}(5) (2002), 727--746.

\bibitem{stepanov}
V.~D. Stepanov, \textit{The weighted Hardy's inequality for
nonincreasing functions}. Trans. Amer. Math. Soc. {\bf 338}(1)
(1993), 173--186.

\bibitem{wedestig}
A. Wedestig, \textit{Weighted inequalities of Hardy type and their
limiting inequalities}. Doctoral Thesis No. 17/03, Dept. of Math.,
Lule{\aa} University of Technology, Sweden, 2003.

\bibitem{wedestig2}
A. Wedestig, \textit{Some new Hardy type inequalities and their
limiting inequalities}. J. Inequal. Pure Appl. Math. {\bf 4}(3)
(2003), Art. 61.
\end{thebibliography}
\end{document}